\documentclass{scrartcl}

\titlehead{Submitted December 2019}

\usepackage{amssymb,amsmath,graphicx,dsfont,xcolor,booktabs}
\usepackage{xcolor}
\usepackage[colorlinks,allcolors=green!60!black]{hyperref}
\usepackage[left=2cm,right=2cm,bottom=3cm,top=3cm]{geometry}

\newcommand{\vt}{\mathbf{v}}
\newcommand{\phit}{\boldsymbol{\phi}}
\newcommand{\omegat}{\boldsymbol{\omega}}

\newcommand{\ft}{\mathbf{f}}

\newcommand{\wt}{\mathbf{w}}

\renewcommand{\div}{\operatorname{div}}

\newlength \figureheight
\newlength \figurewidth
\usepackage{tikz}
\usepackage{pgfplots}

\pgfplotsset{compat=1.9}
\usetikzlibrary{
  pgfplots.groupplots,
  matrix
}

\usepackage{amsmath,amsthm,amssymb}
\newtheorem{lemma}{Lemma}

\newtheorem{algorithm}[lemma]{Algorithm}
\newtheorem{remark}[lemma]{Remark}
\newtheorem{problem}[lemma]{Problem}

\title{An averaging scheme for the efficient approximation of time-periodic
  flow problems}
\author{Thomas Richter
  \thanks{%
    University of Magdeburg, 
    Institute for Analysis and Numerics, 
    Universit\"atsplatz 2, 
    39104 Magdeburg, 
    Germany, 
    \texttt{thomas.richter@ovgu.de}}}

\begin{document}

\maketitle

\begin{abstract}
  We study periodic solutions to the Navier-Stokes equations. The
  transition phase of a dynamic Navier-Stokes solution to the
  periodic-in-time state can be excessively long and it depends on
  parameters like the domain size and the viscosity. 
  Several methods for an accelerated identification of the correct
  initial data that will yield the periodic state exist. They are
  mostly based on space-time frameworks for directly computing the
  periodic state or on optimization schemes or shooting methods for
  quickly finding the correct initial data that yields the periodic
  solution. They all have a large computational overhead in
  common. Here we describe and analyze a simple averaging scheme that
  comes at negligible additional 
  cost. We numerically demonstrate the efficiency and robustness of
  the scheme for several test-cases and we will theoretically show
  convergence for the linear Stokes problem.
\end{abstract}

\section{Introduction}

We study periodic solutions to the Navier-Stokes equations. Such
solutions appear in the laminar regime for the flow around an
obstacle but they can also be induced by a periodic forcing. Here
we investigate this second case which is easier as the frequency of
oscillation is known from the problem data. In the following we will
consider the incompressible Navier-Stokes equations on a domain
$\Omega\subset\mathds{R}^d$ ($d=2,3$)
\begin{equation}\label{NS}
  \div\,\vt = 0,\quad \partial_t\vt+(\vt\cdot\nabla)\vt - \nu \Delta
  \vt+\nabla p = \ft\text{ in }\mathds{R}_+\times \Omega,\quad
  \vt=\vt_0\text{ on }\{0\}\times \Omega,\quad
  \vt=0\text{ on }\mathds{R}_+\times \partial\Omega,
\end{equation}
where $\vt$ is the velocity, $p$ the pressure, $\nu>0$ the viscosity
and $\ft$ the right hand side, which we assume to be $P$-periodic in
time 
\begin{equation}\label{RHS:PER}
  \ft(t+P) = \ft(t),
\end{equation}
where $P>0$ is a fixed period. We are looking for periodic solutions
to~(\ref{NS}) that satisfy $\vt(t+P)=\vt(t)$ and $p(t+P)=p(t)$. In
general, two strategies for identifying such periodic solutions exist:
starting with arbitrary initial values $\vt_0$ one lets the system run
into the periodic state for $t\to\infty$, or, one tries to identify the
correct initial value $\vt_0$ that will directly give the periodic
state.

The necessity to compute such cyclic states of the Navier-Stokes
equations  arises e.g. in the context of temporal multi-scale schemes,
where oscillatory periodic short-scale solutions guide efficient
time-stepping schemes that govern the long-scale
dynamics~\cite{CrouchOskay2015,SandersVerhulstMurdock2007,FreiRichterWick2016,FreiRichter2019}.
Another application is found in 
optimization and steering processes like in simulated moving bed
processes in chemical
engineering~\cite{PlatteKuzminFredebeulTurek2005,LuebkeSeidelMorgensternTobiska2007,ZuyevSeidelMorgensternBenner2016}.

The transition phase of a dynamic Navier-Stokes solution to the
periodic state can be long and it depends, usually
exponentially, on the domain size, on problem
parameters like the viscosity or on discretization parameters like
temporal and spatial resolutions.

Several methods for an acceleration exist. One approach is based on 
a space-time framework for directly computing the periodic
state~\cite{PlatteKuzminFredebeulTurek2005,SteihUrban2012}. This
transforms the problem into a higher-dimensional one with 
substantial numerical overhead. Another possibility is to find
efficient ways for identifying the correct initial data $\vt_0$ for
the velocity $\vt(0)=\vt_0$ 
that gives  a periodic temporal solution of period $P>0$  with
$\vt(t+P)=\vt(t)$ for all $t\ge 0$. The classical approach is the
shooting method and it  has  been demonstrated for systems of ordinary
differential
equations~\cite{RooseLustChampneysSpence1995,JianBieglerFox2003}. Without
special  adaptions as discussed by these authors, the shooting
method 
requires the solution of a high-dimensional problem which comes at
significant costs if partial differential equations are
considered. Yet another approach casts the task of finding a periodic
solution into the identification  of the initial value
$\vt_0$ that solves the nonlinear problem  $F(\vt_0)=0$ where
$F(\vt_0):= \vt_{\vt_0}(P)-\vt_0$, where $\vt_{\vt_0}(t)$ is the
dynamic solution starting with $\vt(0)=\vt_0$. Applying Newton's
method to this problem 
yields a similar structure as the shooting method. The authors
of~\cite{PotschkaMommerSchloederBock2012,HanteMommerPotschka2015}
derived special preconditioners to avoid the large effort of
high-dimensional problems within the Newton scheme.
Nonlinear parabolic problems are analyzed in~\cite{Pao2001} and
monotone iterative schemes are presented for converging to periodic 
solutions. 
Finding the zero
of $F(\vt_0)=0$ can also be tackled as an optimization scheme for
minimizing $\|F(\vt_0)\|^2\to 0$. The authors
of~\cite{AmbroseWilkening2010,RichterWollner2018} apply different
optimization schemes to accelerate this problem. This approach
requires the solution of backward in time adjoint problems.
Finally, the authors of~\cite{LuebkeSeidelMorgensternTobiska2007}
propose an acceleration tool for the forward simulation 
based on a cascadic multilevel method. While this approach requires
the forward simulation only, it might still suffer from a long
transient phase.

Here, an accelerated scheme for the forward simulation is presented
that is based on updating the initial values using the solution of a
stationary  auxiliary
problem. For the Stokes equations we give a proof of the robust
convergence of the scheme with a rate that does not depend on problem
or discretization 
parameters. We numerically demonstrate the efficiency of this scheme
for the nonlinear  Navier-Stokes case.

In the following section we shortly introduce the Navier-Stokes
equations and the notation required for specifying cyclic
states. Section~\ref{sec:projection} presents the averaging algorithm
for accelerating convergence to such periodic states. We give a
complete analysis for the linear case of the Stokes equations and some
hints on treating the nonlinear Navier-Stokes equations. In
Section~\ref{sec:num} we present different numerical test cases
describing the robustness and efficiency of the suggested scheme. We
conclude in Section~\ref{sec:conclusion} with a short outlook to open
problems. 

\section{Periodic-in-time flow problems}\label{sec:projection}

For the following, let $\Omega\subset\mathds{R}^d$ be a domain ($d=2$
or $d=3$). By $H^s(\Omega)$, for $s\in\mathds{N}$ with $s\ge
1$ we denote the Sobolev space of $L^2$ functions on $\Omega$ with
$s$-th weak derivative in $L^2$. By $H^1_0(\Omega)$ we denote the space of
$H^1$-functions with trace zero on the boundary
$\partial\Omega$. By $\|\cdot\|$ we denote
the spatial $L^2$-norm. Further norms are specified by a corresponding
index. 
We start by discussing the linear Stokes equations
\begin{problem}[Stokes]\label{problem:stokes}
  Let $\Omega\subset\mathds{R}^d$ for $d=2,3$ be a domain with a
  boundary that is either smooth ($C^2$-parametrization) or polygonal
  with convex corners. Let $P>0$ be the fixed period and
  \[
  \ft\in  L^\infty(\mathds{R};H^{-1}(\Omega)^d) 
  \]
  with $\ft(t,\cdot)=\ft(t+P,\cdot)$.
  For $\vt_0\in L^2(\Omega)$ let
  \begin{equation}\label{ST}
  \div\,\vt=0,\quad \partial_t \vt-\nu \Delta \vt+\nabla p = \ft
  \text{ in }\mathds{R}_+\times \Omega,\quad
  \vt=\vt_0\text{ on }\{0\}\times \Omega,\quad
  \vt=0\text{ on }\mathds{R}_+\times \partial\Omega
  \end{equation}
  be the solution to the Stokes equations for $\nu>0$. 
\end{problem}
Standard existence and regularity results~\cite{Temam2000} give a
unique solution to Problem~\ref{problem:stokes} that satisfies
\begin{equation}\label{energy}
  \|\vt(t)\|^2 + \int_0^t\nu
  \|\nabla\vt(s)\|^2\,\text{d}s \le \|\vt_0\|^2 + \int_0^t
  \nu^{-1}\|\ft(s)\|_{H^{-1}}^2\,\text{d}s\quad\forall t\in
  \mathds{R}_+. 
\end{equation}
By Poincar\'e's inequality $\|\vt\|\le c_p \|\nabla
\vt\|$ we obtain an integral equation for $\|\vt(t)\|^2$
\begin{equation}\label{energy1}
\|\vt(t)\|^2+\int_0^t\frac{\nu}{c_p^2}\|\vt(s)\|^2\,\text{d}s
\le \|\vt_0\|^2 + \int_0^t\frac{1}{\nu} \|\ft(s)\|_{-1}^2\,\text{d}s,
\end{equation}
which shows that the solution is bounded for all $t\ge 0$ 
\[
\|\vt(t)\|^2 \le
\|\vt_0\|^2\exp\left(-\frac{\nu}{c_p^2}t\right)
+\frac{c_p^2}{\nu^2}\sup_{s\in [0,t]} \|\ft\|_{-1}^2.
\]
Next, let $\wt(t):=\vt(t+P)-\vt(t)$.  This function satisfies the
homogeneous version of equations~(\ref{energy}) and~(\ref{energy1}) such
that it holds 
\begin{equation}\label{perioddecay}
  \|\wt(t)\| \le \|\vt(P)-\vt(0)\|^2
  \exp\left(-\frac{\nu}{c_p^2}t\right)
  \xrightarrow[t\to\infty]{} 0.
\end{equation}
Hence, the solution $\vt(t)$ runs into a unique periodic-in-time state
for $t\to\infty$. Further~(\ref{perioddecay}) shows the
potentially slow decay to the periodic solution depending on the
viscosity $\nu$ and the Poincar\'e constant $c_p$.

These simply results do not carry over to the Navier-Stokes
equations. This is mainly due to the small-data assumption that is
required for the existence of global solutions and the uniqueness of
stationary solutions, which is also of relevance for identifying
periodic-in-time states.
\begin{problem}[Navier-Stokes]\label{problem:navierstokes}
  Let $\Omega\subset\mathds{R}^d$ for $d=2,3$ be a domain with a
  boundary that is either smooth ($C^2$-parametrization) or polygonal
  with convex corners. Let $P>0$ be the fixed period and
  \[
  \ft\in   L^\infty(\mathds{R};H^{-1}(\Omega)^d) 
  \]
  with $\ft(t,\cdot)=\ft(t+P,\cdot)$.
  For $\vt_0\in L^2(\Omega)$ let
  \begin{equation}\label{ST}
  \div\,\vt=0,\quad \partial_t \vt+(\vt\cdot\nabla)\vt-\nu \Delta
  \vt+\nabla p = \ft 
  \text{ in }\mathds{R}_+\times \Omega,\quad
  \vt=\vt_0\text{ on }\{0\}\times \Omega,\quad
  \vt=0\text{ on }\mathds{R}_+\times \partial\Omega
  \end{equation}
  be the solution to the Navier-Stokes equations for $\nu>0$. 
\end{problem}

It has been shown by Kyed and
Galdi~\cite{Kyed2012,GaldiKyed2016} that the solution runs into a
periodic-in-time state $(\vt^\pi,p^\pi)$ if the data is   
sufficiently small, i.e. if the right hand side $\ft$, the initial
value $\vt_0$, the non-homogenous Dirichlet data $\vt^D$ as well as the
period length $P$ are small and the viscosity $\nu$ is sufficiently
large. 

If the proper initial  $\vt_0^\pi$ is known, one cycle of the 
Stokes or Navier-Stokes equations on $[0,P]$ will directly give the
periodic solution. If  however the exact initial value is not given it
might take a tremendous number of cycles (of length $P$) to
sufficiently reduce the 
periodicity error $\|\vt^\pi(t)-\vt(t)\|$.
Here we describe an acceleration approach that is based on projecting
the approximated solution to one that already satisfies a correct
temporal average condition. While the analysis is rigorous for the
linear Stokes equations it remains an heuristic computational
acceleration scheme in the nonlinear case. In
contrast to the various approaches presented in the previous section,
our averaging scheme only requires the solution of one linear 
stationary problem in each cycle as computational overhead. 

\section{Averaging scheme for identifying periodic
  solutions}\label{sec:projection} 

A very slow decay to the periodic solution must be expected in the general
case. Convergence rates
\begin{equation}\label{decay}
\|\vt(t)-\vt^\pi(t)\|
\le C \exp\left(-\frac{\nu}{c_p^2}t\right) \|\vt_0-\vt^\pi_0\|
\end{equation}
are also numerically observed, both for the Stokes and the
Navier-Stokes equations, see Section~\ref{sec:num} for a
numerical illustration.

In this section we will derive an averaging scheme for 
accelerating the convergence to the periodic solution $\vt^\pi(t)$ for
arbitrary initial values $\vt_0$. This averaging scheme is based on
a splitting of the solution into its average and the
fluctuations. On the interval $I=[0,P]$ we introduce
\[
\vt(t):=\bar \vt+\tilde\vt(t),\quad
\bar\vt:=\frac{1}{P}\int_0^P \vt(s)\,\text{d}s,\quad
\tilde\vt(t):=\vt(t)-\bar\vt. 
\]
To speedup the process of finding the periodic solution we try to
quickly adapt the average $\bar\vt$. Averaging the Navier-Stokes
equation, Problem~\ref{problem:navierstokes}, gives
\begin{equation}\label{separate}
\begin{aligned}
  \frac{\vt(P)-\vt(0)}{P}
  +(\bar\vt\cdot\nabla)\bar\vt
  +\frac{1}{P}\int_0^P
  (\tilde\vt(s)\cdot\nabla)\tilde\vt(s)\,\text{d}s 
  -\nu \Delta\bar\vt
  +\nabla\bar p &= \bar \ft,\quad\div\,\bar\vt=0\\
  \partial_t \tilde\vt(t)
  +(\tilde\vt(t)\cdot\nabla)\bar\vt
  + (\bar\vt\cdot\nabla)\tilde\vt(t)
  -\nu \Delta\tilde\vt(t)
  +\nabla\tilde p(t) &= \tilde \ft(t),\quad \div\,\tilde\vt(t)=0,
\end{aligned}
\end{equation}
and reveals that average and fluctuation are coupled and cannot be
computed separately. Assuming that $\vt$ is the periodic solution
$\vt^\pi$ with $\vt^\pi(t)=\vt^\pi(t+P)$ it holds
\[
\begin{aligned}
  (\bar\vt^\pi\cdot\nabla)\bar\vt^\pi +
  \frac{1}{P}\int_0^P(\tilde\vt^\pi(s)\cdot\nabla)\tilde\vt^\pi(s)\,\text{d}s 
  -\nu \Delta\bar\vt^\pi
  +\nabla\bar p^\pi &= \bar \ft,\quad\div\,\bar\vt^\pi=0\\
  \partial_t \tilde\vt^\pi(t)
  +(\tilde\vt^\pi(t)\cdot\nabla)\bar\vt^\pi
  + (\bar\vt^\pi\cdot\nabla)\tilde\vt^\pi(t)
  -\nu \Delta\tilde\vt^\pi(t)
  +\nabla\tilde p^\pi(t) &= \tilde \ft(t),\quad \div\,\tilde\vt^\pi(t)=0,
\end{aligned}
\]
In the averaging scheme we aim at correcting the initial value such
that the correct average
$\tilde\vt^\pi$  is well approximated. For this, we introduce a
problem for the correction $(\bar\wt,\bar q):=(\bar\vt^\pi-\bar\vt,\bar
p^\pi-\bar p)$, i.e. for the difference between periodic solution and
current approximation
\begin{multline}\label{eq:update}
  (\bar\wt\cdot\nabla)\bar\wt
  +(\bar\wt\cdot\nabla)\bar\vt
  +(\bar\vt\cdot\nabla)\bar\wt
  -\nu \Delta\bar\wt + \nabla\bar q\\
  =\frac{\vt(P)-\vt(0)}{P}+
  \frac{1}{P}\int_0^P\Big( (\tilde\vt(s)\cdot\nabla)\tilde\vt(s)
  - (\tilde\vt^\pi(s)\cdot\nabla)\tilde\vt^\pi(s)\Big)\,\text{d}s,
  \quad \div\,\bar\wt=0.
\end{multline}
Naturally, the fluctuation term on the right hand side cannot be
computed without knowledge of the periodic fluctuations
$\tilde\vt^\pi(s)$. However since $\tilde\vt$ shall eventually
converge to $\tilde\vt^\pi$, and since adequate initial values will be
required anyway, we simplify the equation by dropping this term. 
Likewise, we will drop the quadratic term
$(\bar\wt\cdot\nabla)\bar\wt$ since this is second order in $\bar\wt$,
which will converge to zero.

With these preparations we can formulate the approximated averaging
scheme for finding periodic-in-time solutions to the Navier-Stokes
equations
\begin{algorithm}[Averaging Scheme for the Navier-Stokes equations]
  \label{algo:navierstokes}
  Let $\vt_0^{(0)}\in L^2(\Omega)^d$ be an initial value. For 
  $l=1,2,\dots$ iterate
  \begin{enumerate}
  \item Solve the Navier-Stokes equation on $[0,P]$
    \[
    \partial_t \vt^{(l)}(t) - (\vt^{(l)}\cdot\nabla)\vt^{(l)}
    -\nu \Delta\vt^{(l)} +\nabla p^{(l)} = \ft,\quad
    \div\,\vt^{(l)}=0,\quad \vt^{(l)}(0) = \vt_0^{(l-1)}. 
    \]
  \item Compute the average in time
    \[
    \bar \vt^{(l)}:=\frac{1}{P}\int_0^P \vt^{(l)}(s)\,\text{d}s
    \]
  \item Compute the approximated stationary update problem
    \[
    (\bar\wt^{(l)}\cdot\nabla)\bar\vt^{(l)}
    +(\bar\vt^{(l)}\cdot\nabla)\bar\wt^{(l)}
    -\nu\Delta\bar\wt^{(l)} + \nabla \bar q^{(l)}
    =\frac{\vt^{(l)}(P)-\vt^{(l)}(0)}{P},\quad
    \div\,\bar\wt^{(l)}=0.
    \]
  \item Update the initial value
    \[
    \vt^{(l)}_0:= \vt^{(l)}(P) + \bar\wt^{(l)}.
    \]
  \end{enumerate}
\end{algorithm}
The main effort of this scheme is still in the computation of the
dynamic problems in Step 1. We will however observe an significant
reduction of cycles required to approximate the periodic solution.


\subsection{Analysis for the Stokes equations}

The application of the averaging scheme to the Stokes equations
significantly simplified the setting. First, this is due to a
separation of average and fluctuations in
equation~(\ref{separate}). This allows to exactly compute the update
problem~(\ref{eq:update}) without further simplifications. Step 3 of
Algorithm~\ref{algo:navierstokes} can be replaced by solving 
\begin{equation}\label{eq:update:stokes}
  -\nu\Delta\bar\wt^{(l)} + \nabla\bar q^{(l)} =
  \frac{\vt^{(l)}(P)-\vt^{(l)}(0)}{P},\quad 
  \div\,\bar\wt^{(l)}  =0. 
\end{equation}
Second, the symmetric Stokes operator allows for a very simple
analysis based on diagonalization. To be precise, under the
assumptions of Problems~\ref{problem:stokes} and
Problem~\ref{problem:navierstokes} there exists an orthonormal basis
of weakly divergence free 
eigenfunctions $\omegat_i\in H^1_0(\Omega)$ and corresponding
eigenvalues $\lambda_i\in\mathds{R}$ for  $i=1,\dots,\infty$ such that 
(see~\cite{Temam2000}) 
\begin{equation}\label{eigenvalue:stokes}
  \begin{aligned}
    (\nu \nabla \omegat_i,\nabla\phit)-(\zeta_i,\div\,\phit) +
    (\div\,\omegat_i,\xi) &= \lambda_i(\omegat_i,\phit)\quad\forall
    \phit\in H^1_0(\Omega),\; \xi\in
    L^2(\Omega)\setminus\mathds{R}\\
    \text{with }(\omegat_i,\omegat_j)_\Omega = \delta_{ij}&
    \text{ and }
    0<\lambda_1 <\lambda_2\le \lambda_3\le \cdots
  \end{aligned}
\end{equation}
for all $i,j\in\mathds{N}$. 
Given $\vt(t)$, we denote the 
expansion in the eigenfunctions as
\[
\vt(t) = \sum_{i\ge 1} v_i(t)\omegat_i,\quad
\vt_i(t) = (\vt(t),\omegat_i)_\Omega.
\]
A corresponding notation is used for further functions like the
right hand side $\ft(t)$, the periodic solution $\vt^\pi(t)$, the
update $\wt(t)$ or the initial data $\vt_0$. 
Hereby the Stokes equation can be diagonalized and a simple decoupled
system of ode's results
\begin{equation}\label{ODE}
  \partial_t v_i(t) - \nu \lambda_i v_i(t) = f_i(t),\quad i=1,2,\dots.
\end{equation}
Let value $v_i(0)=v_{i,0}\in\mathds{R}$ be the coefficients of the
initial values, the solution
to~(\ref{ODE}) is given by
\begin{equation}\label{SOL}
  v_i(t) = \exp\big(-\lambda_i\nu t) v_i(0) + \int_0^t f_i(s)\exp\big(
  -\lambda_i\nu(t-s)\big)\,\text{d}s. 
\end{equation}

\begin{lemma}[Averaging Scheme for the Stokes
    equation]\label{lemma:ode}
  Let $\Omega\in\mathds{R}^d$ for $d=2$ or $d=3$ be a domain with
  either convex polygonal or smooth ($C^2$-parametrization) boundary,
  $\ft\in L^\infty([0,P],H^{-1}(\Omega)^d)$ and $\vt_0\in
  L^2(\Omega)$. The
  averaging scheme applied to the Stokes equations converges like
  \[
  \|\vt^{(l)}_0 - \vt^\pi_0\|\le 0.3\cdot \|\vt^{(l-1)}_0-\vt^\pi_0\|, 
  \]
  where $\vt^\pi_0$ is the initial value that yields the
  periodic-in-time solution. 
\end{lemma}
\begin{proof}
  Let $v_{i,0}$ be the coefficients of the initial value $\vt_0$, by
  $v_i(t)$ we denote the coefficient functions of the dynamic solution
  on $I=[0,P]$, by $v_i^\pi(t)$ the coefficient functions of the
  unknown periodic-in-time solution. We analyze one iteration of
  Algorithm~\ref{algo:navierstokes} applied to the Stokes
  equations. Hence, $l=1$ and we will drop this index. For the error 
  $w_i(t):=v_i^\pi(t)-v_i(t)$ equation~(\ref{SOL}) gives
  \begin{equation}\label{S1}
    v_i^\pi(t)-v_i(t) = \exp\big(-\lambda_i\nu t\big) (v_{i,0}^\pi-v_{i,0})
  \end{equation}
  Step 2 can be omitted since the average does not enter the update
  equation in the linear case, compare~(\ref{eq:update:stokes}). 
  The solution to the stationary update equation of Step 3
  is given by 
  \begin{equation}\label{S2}
    \lambda_i \nu \bar w_i = \frac{v_i(P)-v_i(0)}{P}
    \quad\Rightarrow\quad \bar w_i =
    \frac{(v_i^\pi(0)-v_i(0))-(v_i^\pi(P)-v_i(P))}{\lambda_i\nu P},
  \end{equation}
  where we used that $v_i^\pi(0)=v_i^\pi(P)$. With~(\ref{S1}) this
  gives
  \begin{equation}\label{S3}
    \bar w_i =\frac{1}{\lambda_i\nu P}\Big( 1- \exp\big(-\lambda_i\nu P\big)\Big) 
    \big(v_i^\pi(0)-v_i(0)\big)
  \end{equation}
  Then, the newly computed initial value from Step 4 in the algorithm
  satisfies
  \begin{equation}\label{S4}
    v_i^\pi- \big(v_i(P) + \bar w_i\big)
    = \big(v_i^\pi(P)-v_i(P)\big) - \bar w_i
    = \Big(\exp\big(-\lambda_i \nu P\big)
    \left(1+\frac{1}{\lambda_i\nu P}\right)
    -\frac{1}{\lambda_i\nu P}\Big)
    \big(v_i^\pi(0)-v_i(0)\big)
  \end{equation}
  where we used~(\ref{S3}) and~(\ref{S1}). Let $s=\lambda_i\nu
  P>0$. We study the function
  \begin{equation}\label{S5}
    \rho(s) = \exp(-s)\left(1+\frac{1}{s}\right)-\frac{1}{s}
    =\frac{\exp(-s)}{s}\big(1+s-\exp(s)\big),
  \end{equation}
  which is negative for $s\in\mathds{R}_+$ since $\exp(s)\ge
  1+s$. Further it holds $\rho(0)=0$ and $\rho(s)\to 0$ for
  $s\to\infty$. $\rho(s)$ has only one extreme point
  in  $\mathds{R}_+$ that is easily found numerically and gives the bound
  $|\rho(x)|<0.299$ for $x\in\mathds{R}_+$. 
\end{proof}
The convergence rate of this averaging scheme is always at least $0.3$
and does in particular not depend on the stiffness rate
$\lambda_i\nu>0$ or the period length $P>0$. 

\section{Discrete setting}\label{disc}

For discretization of the dynamic Navier-Stokes and Stokes equations we
employ standard techniques which we summarize briefly. Details are
found in~\cite[Sections 4.1 and
  4.2]{Richter2017}. All implementations are done in Gascoigne
3D~\cite{Gascoigne3D}. In time we use 
the $\theta$-time-stepping scheme that can be considered as a variant of
the Crank-Nicolson
scheme~\cite{LuskinRannacher1982,HeywoodRannacher1990} with better
smoothing properties capable of giving robust long time solutions. Let
$k=P/N$ be the time step size and $t_n = nk$ for $n=0,1,\dots$ be the
uniform partitioning of $[0,P]$.
By $\vt_n\approx \vt(t_n)$ and $p_n\approx p(t_n)$ we denote the  
approximation to the solution at time $t_n$. Likewise
$\ft^n:=\ft(t_n)$ is the right hand side evaluated at time $t_n$. 

Let $V_h\times Q_h\subset
H^1_0(\Omega;\Gamma^D)^d \times L^2(\Omega)$ be a suitable 
finite element pair. For simplicity assume that $V_h\times Q_h$ is an
inf-sup stable 
pressure-velocity pair like the Taylor-Hood
element~\cite{Brezzi1991}. Alternatively, 
stabilized equal order finite elements, e.g. based on the local
projection stabilization scheme~\cite{BeckerBraack2001}, can be
used. This is indeed the standard setting of our implementation in
Gascoigne 3D~\cite{Gascoigne3D}. The 
complete space-time discrete formulation of the incompressible
Navier-Stokes equation is given by
\begin{multline}\label{fd}
  \vt^n\in V_h,\quad p^n\in V_h,\quad \vt^0:=\vt_0,\quad
  \text{for }n=1,2,\ldots\\ 
  \Big(\vt^n-\vt^{n-1},\phi\Big) + k\Big(
  (1-\theta)(\vt^{n-1}\cdot\nabla)\vt^{n-1} + \theta
  (\vt^n\cdot\nabla)\vt^n,\phi\Big)\\
  +   \nu k\left( \nabla \left(
  (1-\theta)\vt^{n-1} + \theta\vt^n\right),\nabla\phi\right)
  -k\Big(p^n,\nabla\cdot\phi\Big)\\
  +k \Big(\nabla\cdot \vt^n,\xi\Big) = k\Big(
  (1-\theta)\ft^{n-1}+\theta\ft^n,\phi\Big)\quad\forall (\phi,\xi)\in
  V_h\times Q_h.
\end{multline}
The  parameter $\theta$ is chosen in $[1/2,1]$.
For $\theta=1$ this
scheme corresponds to the backward Euler method, for $\theta=1/2$ to
the Crank-Nicolson scheme and for $\theta>1/2$ to the shifted
Crank-Nicolson scheme which has better smoothing properties. For
$\theta=1/2+{\mathcal O}(k)$ we get global stability and still have
second order convergence,
see~\cite{LuskinRannacher1982,Rannacher1984,RichterWick2015_time}. In
addition to~(\ref{fd}) we also consider the discretization of the Stokes
equation, which is realized by skipping the convective terms. 

To transfer the averaging scheme to the discrete setting we start by
indicating a discrete counterpart of the averaging operator that is
conforming with the 
discretization based on the $\theta$-scheme. We sum~(\ref{fd}) over
all time steps using the same pair of test functions $(\phi,\xi)$ for all
steps and divide by the period $P$
\begin{equation}\label{summed}
\frac{1}{P}\Big(\vt^N-\vt^{0},\phi\Big) + \Big(
\overline{(\vt\cdot\nabla)\vt}^k ,\phi\Big)
+   \nu \left( \nabla \overline{\vt}^k,\nabla\phi\right) 
-\Big(\overline{p}^{k,0},\nabla\cdot\phi\Big)
+ \Big(\nabla\cdot \overline{\vt}^{k},\xi\Big) = \Big(
  \overline{\ft}^k,\phi\Big)
\end{equation}
with the discrete averaging operators
\begin{equation}\label{avgdic}
  \overline{\vt}^k:=\frac{k}{P}\sum_{n=1}^N
  \Big((1-\theta)\vt_{n-1} + \theta\vt_n\Big),\quad
  \overline{p}^{k,0}:=\frac{k}{P}\sum_{n=1}^{N} p_n. 
\end{equation}
Directly summing the divergence term in~(\ref{fd}) yields the average
$(\nabla\cdot \overline{\vt}^{k,0},\xi)=0$, this however is equivalent
to $(\nabla\cdot \overline{\vt}^{k,0},\xi)=(\nabla\cdot
\overline{\vt}^{k},\xi)$. The discrete periodic solution
$(\vt^\pi,p^\pi)$ satisfies the equation
\begin{equation}\label{summed:avg}
  \Big(
  \overline{(\vt^\pi\cdot\nabla)\vt^\pi}^k ,\phi\Big)
  +   \nu \left( \nabla \overline{\vt^\pi}^k,\nabla\phi\right) 
-\Big(\overline{p^\pi}^{k,0},\nabla\cdot\phi\Big)
+ \Big(\nabla\cdot \overline{\vt^\pi}^{k},\xi\Big) = \Big(
  \overline{\ft}^k,\phi\Big),
\end{equation}
such that the approximated discrete update equation, i.e. the equation
for approximating the difference
$(\overline{\wt}^k,\overline{q}^{k,0}) \approx
(\overline{\vt^\pi}^{k}-\overline{\vt}^k,\overline{p^\pi}^{k,0} -
\overline{p}^{k,0})$ is given by 
\begin{equation}\label{summed:up}
  \Big(
  (\overline{\wt}^k\cdot\nabla)\overline{\vt}^k
  +(\overline{\vt}^k\cdot\nabla)\overline{\wt}^k ,\phi\Big)
  + \nu \left( \nabla \overline{\wt}^k,\nabla\phi\right) 
  -\Big(\overline{q}^{k,0},\nabla\cdot\phi\Big)
  + \Big(\nabla\cdot \overline{\wt}^{k},\xi\Big) =
  \frac{1}{P}\Big(\vt_N-\vt_0,\phi\big),
\end{equation}
where the same approximations are applied as in the continuous case:
we neglect the averaging error $\overline{(\wt\cdot\nabla)\wt}^k-
(\overline{\wt}^k\cdot\nabla)\overline{\wt}^k$ as well as the quadratic
nonlinearity  $(\overline{\wt}^k\cdot\nabla)\overline{\wt}^k$.

\begin{algorithm}[Discrete averaging Scheme for the Navier-Stokes
    equations] \label{algo:disc}
  Let $\vt_0^{(0)}\in V_h$ be an initial value. For
  $l=1,2,\dots$ iterate
  \begin{enumerate}
  \item Solve the Navier-Stokes~(\ref{fd}) equation on $[0,P]$ for
    $(\vt^{(l)}_n,p^{(l)}_n)$, $n=1,\dots,N$ with
    $\vt^{(l)}_0=\vt_0^{(l-1)}$. 
  \item Compute $\overline{\vt^{(l)}}^k$, the discrete average in
    time, by~(\ref{summed:avg}). 
  \item Compute the approximated (stationary) update
    problem~(\ref{summed:up}) for
    $\overline{\wt^{(l)}}^k,\overline{q^{(l)}}^{k,0}$. 
  \item Update the initial value
    \[
    \vt^{(l)}_0:= \vt^{(l)}_N + \overline{\wt^{(l)}}^k.
    \]
  \end{enumerate}
\end{algorithm}

Since we consider conforming finite element discretizations the proof
for the linear case can be transferred to the discrete setting. Some
care is required to obtain sufficient stability of the time stepping
scheme.
\begin{lemma}[Discrete averaging Scheme for the Stokes
    equation]\label{lemma:ode:disc}
  Let $\Omega\in\mathds{R}^d$ for $d=2$ or $d=3$ be a domain with
  either convex polygonal or smooth ($C^2$-parametrization) boundary,
  $\ft\in 
  L^\infty([0,P],H^{-1}(\Omega)^d)$ and $\vt_0\in L^2(\Omega)$. Let
  $V_h\times Q_h\subset H^1_0(\Omega)^d\times L^2(\Omega)$ be an
  inf-sup stable finite element space. Let $I=[0,P]$ be discretized in
  $N\ge 4$ time steps of size $k=P/N$ with time-scheme parameter
  \[
  \theta_N = \frac{1}{2}+\frac{1}{2N}.
  \]
  Then, the
  discrete averaging Algorithm~\ref{algo:disc} applied to the Stokes
  equations converges like 
  \[
  \|\vt^{(l)}_0 - \vt^\pi_0\|\le 0.42\cdot \|\vt^{(l-1)}_0-\vt^\pi_0\|.
  \]
\end{lemma}
\begin{proof}
  Let $(\omegat_i,\mu_i,\lambda_i)\in V_h\times Q_h\times \mathds{R}$
  for $i=1,\dots,N_h:=\operatorname{div}(V_h)$ be a system of  
  $L^2$-orthonormal, discretely divergence free eigenfunctions and
  eigenvalues of the discrete Stokes operator. By
  $v_{i,0}, f_i\in \mathds{R}$ and $v_i(t)\in\mathds{R}$ we denote the
  coefficient (function) of an initial value $\vt_0$ of the right hand
  side and of the solution $\vt(t)$ with respect to this
  basis. Comparable to the continuous case we analyze one single step
  of the discrete averaging scheme given in
  Algorithm~\ref{algo:disc}.

  \medskip\noindent \emph{(i)}
  The discretized Stokes equation can be decoupled into a system of
  $N_h$ difference equations
  \[
  v^n_i-v^{n-1}_i + \nu k \lambda_i \big( (1-\theta)v^{n-1}_i + \theta
  v^n_i\big) =
  \nu k  \big( (1-\theta)f^{n-1}_i + \theta
  f^n_i\big),\quad i=1,\dots,N_h. 
  \]
  We measure the difference $v^{\pi,n}_i-v_i^n$ between the 
  solution to the correct initial value $v^\pi_{0,i}$ and to an arbitrary starting
  value $v_{0,i}$ 
  \begin{equation}\label{decay:disc}
    v^{\pi,n}_i-v^n_i=\underbrace{\frac{1-\nu k
        (1-\theta)\lambda_i}{1+\nu k \theta 
        \lambda_i}}_{=:q_i}  w^{n-1}_i = q_i^n (v^\pi_{0,i}-v_{0,i}). 
  \end{equation}
  The coefficients of the stationary update equation in Step 3 of
  Algorithm~\ref{algo:disc} are determined by
  \[
  \overline{w_i}^k =\frac{v_i^N-v_i^0}{\nu\lambda_i P}
  \]
  and with $v^{\pi,N}_i=v^{\pi,0}_i=v^\pi_{0,i}$ the new initial
  computed in Step 4 of Algorithm~\ref{algo:disc} is
  carrying the error
  \[
  v^\pi_{0,i}-\left(v^N_i + \overline{w_i}^k\right)
  =v^{\pi,N}_i - v^N_i + \frac{v_i^{\pi,N}-v_i^N
    -(v_i^{\pi,0}-v_i^0)}{\nu\lambda_i P},
  \]
  which, together with~(\ref{decay:disc}), is estimated as
  \[
  v^\pi_{0,i}-\left(v^N_i + \overline{w_i}^k\right)
  =\left(q_i^N\left(1+\frac{1}{\nu\lambda_i
    P}\right)-\frac{1}{\nu\lambda_i P}\right) 
  (v^\pi_{0,i}-v_{0,i}). 
  \]
  With $s_i:=\nu\lambda_i P\in [s_1,s_{|V_h|}]\subset (0,\infty)$ we
  identify the reduction factor
  \begin{equation}\label{disc:1}
    \rho^N(s):=\left(1-\frac{s}{N+\theta
      s}\right)^N\left(1+\frac{1}{s}\right)
    -\frac{1}{s},
  \end{equation}
  which for $N\to \infty$ converges to the reduction rate $\rho(s)$
  of the continuous case, see~(\ref{S5}).

  \medskip\noindent
  \emph{(ii)} For estimating~(\ref{disc:1}) we consider two
  cases. First, let $0\le s\le N$. Then, it holds
  \begin{equation}\label{disc:2}
    0\le 1-\frac{s}{N}\le 
    1-\frac{s}{N+\theta s}\le 1-\frac{s}{N(1+\theta)}
  \end{equation}
  and hereby, we can estimate $\rho^N(s)$ to both sides by
  \[
  \begin{aligned}
  \left(1-\frac{s}{N}\right)^N\left(1+\frac{1}{s}\right)
  -\frac{1}{s} &\le 
  \rho^N(s) \le
  \left(1-\frac{s}{N(1+\theta)}\right)^N\left(1+\frac{1}{s}\right)
  -\frac{1}{s}\\
  \Leftrightarrow\qquad
  \exp\left(-s\right)\left(1+\frac{1}{s}\right)
  -\frac{1}{s} &\le 
  \rho^N(s) \le
  \exp\left(-\frac{s}{1+\theta}\right)\left(1+\frac{1}{s}\right)
  -\frac{1}{s}.
  \end{aligned}
  \]
  The lower bound is $\rho(s)$, see Eq.~(\ref{S5}), with $-0.29<\rho(s)$
  for all $s\ge 0$. The upper bound takes its maximum at $s\to 0$ 
  \[
  \exp\left(-\frac{s}{1+\theta}\right)\left(1+\frac{1}{s}\right)
  -\frac{1}{s}
  \xrightarrow[s\to 0]{} \frac{\theta}{1+\theta}
  \]
  such that it holds
  \[
  -0.3 \le \rho^N(s) \le \frac{\theta}{1+\theta},\quad \forall 0\le
  s\le N. 
  \]
  For $\theta=\theta_N$ and $N\ge 4$ this gives the bound
  \[
  |\rho^N\big|_{\theta=\theta_N}| \le 0.39\quad \forall 0\le s\le N. 
  \]
  Next, let $s\ge N$. The function
  \[
  q(s):=1-\frac{s}{N+\theta s}
  \]
  is monotonically decreasing, such that for $s\ge N$ it holds
  \[
  1-\frac{1}{\theta}\le q(s)= 1-\frac{s}{N+\theta s} \le
  1-\frac{1}{1+\theta}
  \quad\Rightarrow\quad |q(s)| \le \frac{1}{\theta}-1.
  \]
  Choosing $\theta=\theta_N$ and $N\ge 2$ gives
  \[
  |q(s)^N| \le \left|\frac{N-1}{N+1}\right|^N=
  \left|1-\frac{2}{N+1}\right|^N\le \exp(-2).
  \]
  Altogether we estimate for $s\ge N$ with $N\ge 4$ we get the bound
  \[
  |\rho^N(s)\Big|_{\theta=\theta_N}| \le
  \exp(-2)\left(1+\frac{1}{N}\right) + \frac{1}{N}
  \le 1.25 \exp(-2)+0.25   \le 0.42.
  \]
\end{proof}

\begin{remark}
  Numerical results show that this estimate is not sharp, convergence
  rates close to $0.3$ are observed (which is the reduction rate in
  the continuous setting). The specific choice of
  \[
  \theta_N = \frac{1}{2}+\frac{1}{2N}
  \]
  corresponds to a slightly shifted version of the Crank-Nicolson
  scheme. It is of second order and has improved stability
  properties. The choice of $\theta_N$ corresponds to  $\theta
  = \frac{1}{2}+\frac{P}{2}k$ and it can be generalized to
  $\theta=\frac{1}{2}+\alpha k$ for $\alpha>0$. 
  In the numerical test cases we observe no problems with
  the standard Crank-Nicolson scheme $\theta=1/2$. 
\end{remark}

\section{Numerical test cases}\label{sec:num}

We discuss different numerical test cases to highlight the
efficiency and robustness of the averaging scheme for the computation
of cyclic states. We directly consider the Navier-Stokes equations but
include problems at very low Reynolds numbers. The linear Stokes
problem gives comparable results. Here,  perfect robustness of the
averaging scheme for arbitrary variations of the viscosity, the domain
size, the velocity, etc. if found. 

Before presenting the specific test cases we shortly describe the
computational setting implemented in the finite element software
library \emph{Gascoigne 3D}~\cite{Gascoigne3D}: Discretization (in 2d)
is based on quadrilateral meshes. To cope with the saddle-point
structure we utilize stabilized quadratic equal-order
elements. Stabilization is based on the local projection
scheme~\cite{BeckerBraack2001}. As the appearing Reynolds numbers are
very moderate we do not require any stabilization of convective terms.
Nonlinear problems are approximated with a Newton scheme using
analytic Jacobians. The equal-order setup allows us to use an
efficient geometric multigrid solver for all linear problems,
see~\cite{BeckerBraack2000a} for the general setup
and~\cite{KimmritzRichter2010} for the efficient implementation.
Although the analysis shows superior robustness for a shifted version
of the Crank-Nicolson scheme with $\theta>1/2$ we do not observe any
difficulties with the choice $\theta=1/2$ which will be used
throughout this section. In~\cite{RichterMizerski2020} an application
of the averaging scheme to the efficient simulation of temporal
multiscale problems is given. Here we also include a study on a three
dimensional test case. These results are in perfect agreement to the
following two dimensional cases. 

\subsection{Robustness of the averaging scheme}\label{sec:num:stokes}

For $L\in\mathds{R}_+$ let $\Omega = 
(-L,L)^2$ and $I=[0,P]$ for a given $P\in\mathds{R}_+$. We solve
\[
\begin{aligned}
  \nabla\cdot\vt =0,\quad   \partial_t \vt+(\vt\cdot\nabla)\vt -\nu \Delta \vt + \nabla p &=
  \ft&& \text{ in }I\times \Omega\\ 
  \vt=\vt_0\text{ on } \{0\}\times \Omega,\quad \vt&=0&&\text{ on } 
  I\times \partial\Omega,
\end{aligned}
\]
and try to identify a time-periodic solution $\vt(P)=\vt(0)$. The
forcing $\ft$ is $P$-periodic and given by
\[
\ft(x,y,t) =
\frac{\tanh\big(y\big)}{LP}
\sin\left(\frac{2\pi 
  t}{P}\right)\begin{pmatrix}1\\ 0 
\end{pmatrix}.
\]

\begin{figure}[t]
  \setlength{\figureheight}{0.3\textwidth}
  \setlength{\figurewidth}{0.4\textwidth} 
  
  \begin{center}
    \begin{tabular}{cc}
      \parbox{0.48\textwidth}{
      \begin{tikzpicture}
  \begin{semilogyaxis}[
      width=\figurewidth, height=\figureheight,
      scale only axis,
      x label style={anchor=north, below=-12mm},
      xlabel=cycles,
      mark options={solid},
      title style={font=\bfseries},
      legend style={at={(\figurewidth,0.95)},xshift=-0.2cm,anchor=north east,nodes=right}]
    \addplot[color=blue,solid,line width=0.5mm] table[row sep=crcr]{
      2	0.0497606	\\
      3	0.0293591	\\
      4	0.00358413	\\
      5	0.00100257	\\
      6	0.000283001	\\
      7	8.00906e-05	\\
      8	2.28839e-05	\\
      9	6.56557e-06	\\
      10	1.88392e-06	\\
      11	5.41209e-07	\\
      12	1.55537e-07	\\
      13	4.46971e-08	\\
      14	1.28441e-08	\\
      15	3.69068e-09	\\
    };
    \addlegendentry{$L=1$}
    \addplot[solid,line width=0.5mm] table[row sep=crcr]{
      2	0.040718	\\
      3	0.0396963	\\
      4	0.0024437	\\
      5	0.000657732	\\
      6	0.000184053	\\
      7	5.26298e-05	\\
      8	1.4954e-05	\\
      9	4.21247e-06	\\
      10 1.18558e-06	\\
      11 3.29752e-07	\\
      12 9.18535e-08	\\
      13 2.53591e-08	\\
      14 7.00737e-09	\\
    };
    \addlegendentry{$L=2$}
    \addplot[solid,color=red,line width=0.5mm] table[row sep=crcr]{
      2	0.0276215	\\
      3	0.0228166	\\
      4	0.00477496	\\
      5	0.00111102	\\
      6	0.000161514	\\
      7	4.17968e-05	\\
      8	1.1146e-05	\\
      9	3.03335e-06	\\
      10	8.33905e-07	\\
      11	2.33751e-07	\\
      12	6.5562e-08	\\
      13	1.83917e-08	\\
      14	5.16003e-09	\\
    };
    \addlegendentry{$L=4$}
    \addplot[mark=*,color=blue,line width=0.2mm,mark size=0.5mm] table[row sep=crcr]{
      2	0.0497606	\\
      3	0.0110588	\\
      4	0.00299763	\\
      5	0.000810536	\\
      6	0.000219065	\\
      7	5.92054e-05	\\
      8	1.6001e-05	\\
      9	4.32449e-06	\\
      10	1.16875e-06	\\
      11	3.15871e-07	\\
      12	8.53682e-08	\\
      13	2.30719e-08	\\
      14	6.23548e-09	\\
    };
    \addlegendentry{$L=1$}
    \addplot[mark=+,mark size=0.5mm,line width=0.2mm] table[row sep=crcr]{
2	0.040718	\\
3	0.0124829	\\
4	0.00731993	\\
5	0.00496711	\\
6	0.0035285	\\
7	0.00253523	\\
8	0.00182434	\\
9	0.00131688	\\
10	0.000950305	\\
11	0.000685452	\\
12	0.000494303	\\
13	0.000356425	\\
14	0.000256996	\\
15	0.000185298	\\
16	0.000133599	\\
17	9.63241e-05	\\
18	6.94486e-05	\\
19	5.00716e-05	\\
20	3.61009e-05	\\
21	2.60283e-05	\\
22	1.8766e-05	\\
23	1.353e-05	\\
24	9.75496e-06	\\
25	7.03319e-06	\\
26	5.07083e-06	\\
27	3.656e-06	\\
28	2.63592e-06	\\
29	1.90046e-06	\\
30	1.37021e-06	\\
31	9.87899e-07	\\
32	7.12261e-07	\\
33	5.1353e-07	\\
34	3.70248e-07	\\
35	2.66943e-07	\\
36	1.92462e-07	\\
37	1.38763e-07	\\
38	1.00046e-07	\\
39	7.21316e-08	\\
40	5.20059e-08	\\
41	3.74955e-08	\\
42	2.70337e-08	\\
43	1.94909e-08	\\
44	1.40527e-08	\\
45	1.01318e-08	\\
46	7.30487e-09	\\
    };
    \addlegendentry{$L=2$}
    \addplot[mark=diamond,mark size=0.5mm,line width=0.2mm,color=red] table[row sep=crcr]{
      2	0.0276215	\\
3	0.00870959	\\
4	0.00505705	\\
5	0.00362721	\\
6	0.00276177	\\
7	0.00227139	\\
8	0.00189834	\\
9	0.00162561	\\
10	0.00142849	\\
11	0.00126236	\\
12	0.00112151	\\
13	0.00101213	\\
14	0.000918037	\\
15	0.00083425	\\
16	0.000759451	\\
17	0.000692535	\\
18	0.000636588	\\
19	0.000585304	\\
20	0.000538287	\\
21	0.000495167	\\
22	0.000455602	\\
23	0.000419283	\\
24	0.000385928	\\
25	0.000355282	\\
26	0.000327115	\\
27	0.000301215	\\
28	0.00027749	\\
29	0.000255762	\\
30	0.000235725	\\
31	0.00021725	\\
32	0.000200217	\\
33	0.000184515	\\
34	0.000170041	\\
35	0.000156699	\\
36	0.000144403	\\
37	0.00013307	\\
38	0.000122625	\\
39	0.000112999	\\
40	0.000104129	\\
41	9.59537e-05	\\
42	8.84203e-05	\\
43	8.14787e-05	\\
44	7.50836e-05	\\
45	6.919e-05	\\
46	6.37586e-05	\\
47	5.87533e-05	\\
48	5.41407e-05	\\
49	4.989e-05	\\
50	4.5973e-05	\\
    };
    \addlegendentry{$L=4$}
  \end{semilogyaxis}
\end{tikzpicture}}
      &
      \parbox{0.48\textwidth}{
      \begin{tikzpicture}
  \begin{semilogyaxis}[
      width=\figurewidth, height=\figureheight,
      scale only axis,
      x label style={anchor=north, below=-12mm},
      xlabel=cycles,
      mark options={solid},
      title style={font=\bfseries},
      legend style={at={(\figurewidth,0.95)},xshift=-0.2cm,anchor=north east,nodes=right}]
    \addplot[color=blue,solid,line width=0.5mm] table[row sep=crcr]{
    2	0.0403474	\\
3	0.00512544	\\
4	0.000550745	\\
5	7.3483e-05	\\
6	9.52856e-06	\\
7	1.25059e-06	\\
8	1.63562e-07	\\
9	2.1416e-08	\\
10	2.80313e-09	\\
    };
    \addlegendentry{$P=4$}
    \addplot[solid,line width=0.5mm] table[row sep=crcr]{
    2	0.0403596	 \\
3	0.0159469	 \\
4	0.000616497	 \\
5	9.7356e-05	 \\
6	1.19473e-05	 \\
7	2.03077e-06	 \\
8	2.80606e-07	\\
9	4.51702e-08	\\
10	6.63728e-09	\\
    };
    \addlegendentry{$P=2$}
    \addplot[solid,color=red,line width=0.5mm] table[row sep=crcr]{
    2	0.040718	\\
3	0.0396963	\\
4	0.0024437	\\
5	0.000657732	\\
6	0.000184053	\\
7	5.26298e-05	\\
8	1.4954e-05	\\
9	4.21247e-06	\\
10	1.18558e-06	\\
11	3.29752e-07	\\
12	9.18535e-08	\\
13	2.53591e-08	\\
14	7.00737e-09	\\
    };
    \addlegendentry{$P=1$}
    \addplot[mark=*,color=blue,line width=0.2mm,mark size=0.5mm] table[row sep=crcr]{
    2	0.0403474	\\
3	0.00872483	\\
4	0.00236091	\\
5	0.000638271	\\
6	0.000172501	\\
7	4.66196e-05	\\
8	1.25992e-05	\\
9	3.40501e-06	\\
10	9.20224e-07	\\
11	2.48696e-07	\\
12	6.72115e-08	\\
13	1.81643e-08	\\
14	4.90901e-09	\\
    };
    \addlegendentry{$P=4$}
    \addplot[mark=+,mark size=0.5mm,line width=0.2mm] table[row sep=crcr]{
2	0.0403596	 \\
3	0.0120208	 \\
4	0.00598796	 \\
5	0.00310265	 \\
6	0.00161608	 \\
7	0.000840482	 \\
8	0.000436977	 \\
9	0.000227165	\\
10	0.000118089	 \\
11	6.13871e-05	 \\
12	3.19111e-05	 \\
13	1.65885e-05	 \\
14	8.62328e-06	 \\
15	4.48268e-06	 \\
16	2.33025e-06	 \\
17	1.21134e-06	 \\
18	6.29699e-07	 \\
19	3.27339e-07	 \\
20	1.70162e-07	\\
21	8.84562e-08	\\
22	4.59826e-08	\\
23	2.39033e-08	\\
24	1.24258e-08	\\
25	6.45934e-09	\\
    };
    \addlegendentry{$P=2$}
    \addplot[mark=diamond,mark size=0.5mm,line width=0.2mm,color=red] table[row sep=crcr]{
2	0.040718	\\
3	0.0124829	\\
4	0.00731993	\\
5	0.00496711	\\
6	0.0035285	\\
7	0.00253523	\\
8	0.00182434	\\
9	0.00131688	\\
10	0.000950305	\\
11	0.000685452	\\
12	0.000494303	\\
13	0.000356425	\\
14	0.000256996	\\
15	0.000185298	\\
16	0.000133599	\\
17	9.63241e-05	\\
18	6.94486e-05	\\
19	5.00716e-05	\\
20	3.61009e-05	\\
21	2.60283e-05	\\
22	1.8766e-05	\\
23	1.353e-05	\\
24	9.75496e-06	\\
25	7.03319e-06	\\
26	5.07083e-06	\\
27	3.656e-06	\\
28	2.63592e-06	\\
29	1.90046e-06	\\
30	1.37021e-06	\\
31	9.87899e-07	\\
32	7.12261e-07	\\
33	5.1353e-07	\\
34	3.70248e-07	\\
35	2.66943e-07	\\
36	1.92462e-07	\\
37	1.38763e-07	\\
38	1.00046e-07	\\
39	7.21316e-08	\\
40	5.20059e-08	\\
41	3.74955e-08	\\
42	2.70337e-08	\\
43	1.94909e-08	\\
44	1.40527e-08	\\
45	1.01318e-08	\\
46	7.30487e-09	\\
    };
    \addlegendentry{$P=1$}
  \end{semilogyaxis}
\end{tikzpicture}}
      \\
      \parbox{0.48\textwidth}{
      \begin{tikzpicture}
  \begin{semilogyaxis}[
      width=\figurewidth, height=\figureheight,
      scale only axis,
      x label style={anchor=north, below=-12mm},
      xlabel=cycles,
      mark options={solid},
      title style={font=\bfseries},
      legend style={at={(\figurewidth,0.95)},xshift=-0.2cm,anchor=north east,nodes=right}]
    \addplot[color=blue,solid,line width=0.5mm] table[row sep=crcr]{
1	0.0403662	
2	0.0192276	\\
3	0.00290778	\\
4	0.000810289	\\
5	0.000228534	\\
6	6.4629e-05	\\
7	1.83877e-05	\\
8	5.27484e-06	\\
9	1.51367e-06	\\
10	4.34493e-07	\\
11	1.24758e-07	\\
12	3.58313e-08	\\
13	1.02932e-08	\\
14	2.95828e-09	\\
};
    \addlegendentry{$\nu=0.4$}
    \addplot[solid,line width=0.5mm] table[row sep=crcr]{
    1	0.0407366	\\
2	0.039687	\\
3	0.00244442	\\
4	0.000658136	\\
5	0.000184287	\\
6	5.2731e-05	\\
7	1.49925e-05	\\
8	4.22616e-06	\\
9	1.19021e-06	\\
10	3.31274e-07	\\
11	9.23423e-08	\\
12	2.55124e-08	\\
13	7.05465e-09	\\
    };
    \addlegendentry{$\nu=0.1$}
    \addplot[solid,color=red,line width=0.5mm] table[row sep=crcr]{
    1	0.0391445	
2	0.0424874	\\
3	0.00351014	\\
4	0.00492496	\\
5	0.000425011	\\
6	0.000155808	\\
7	3.34941e-05	\\
8	1.17075e-05	\\
9	2.97373e-06	\\
10	9.19202e-07	\\
11	2.08771e-07	\\
12	6.20127e-08	\\
13	1.72378e-08	\\
14	6.08189e-09	\\
    };
    \addlegendentry{$\nu=0.025$}
    \addplot[mark=*,color=blue,line width=0.2mm,mark size=0.5mm] table[row sep=crcr]{
2       	0.0403662	\\
3       	0.00872641	\\
4       	0.00236024	\\
5       	0.000637834	\\
6       	0.000172317	\\
7       	4.65524e-05	\\
8       	1.25764e-05	\\
9       	3.39757e-06	\\
10	9.17871e-07	\\
11	2.47967e-07	\\
12	6.69896e-08	\\
13	1.80976e-08	\\
14	4.88915e-09	\\
};
    \addlegendentry{$\nu=0.4$}
    \addplot[mark=+,mark size=0.5mm,line width=0.2mm] table[row sep=crcr]{
2       	0.0407366	\\
3       	0.0124839	\\
4       	0.00732031	\\
5       	0.00496733	\\
6       	0.00352864	\\
7       	0.00253527	\\
8       	0.00182431	\\
9       	0.00131684	\\
10	0.000950246	\\
11	0.00068539	\\
12	0.000494245	\\
13	0.000356374	\\
14	0.000256953	\\
15	0.000185261	\\
16	0.000133569	\\
17	9.62999e-05	\\
18	6.94293e-05	\\
19	5.00563e-05	\\
20	3.6089e-05	\\
21	2.60189e-05	\\
22	1.87588e-05	\\
23	1.35245e-05	\\
24	9.75068e-06	\\
25	7.02992e-06	\\
26	5.06833e-06	\\
27	3.6541e-06	\\
28	2.63448e-06	\\
29	1.89937e-06	\\
30	1.36938e-06	\\
31	9.8728e-07	\\
32	7.11796e-07	\\
33	5.13181e-07	\\
34	3.69986e-07	\\
35	2.66748e-07	\\
36	1.92316e-07	\\
37	1.38654e-07	\\
38	9.99645e-08	\\
39	7.20711e-08	\\
40	5.19608e-08	\\
41	3.7462e-08	\\
42	2.70089e-08	\\
43	1.94725e-08	\\
44	1.4039e-08	\\
45	1.01217e-08	\\
46	7.29737e-09	\\
    };
    \addlegendentry{$\nu=0.1$}
    \addplot[mark=diamond,mark size=0.5mm,line width=0.2mm,color=red] table[row sep=crcr]{
2       	0.0391445	\\
3       	0.0124471	\\
4       	0.00735842	\\
5       	0.00530438	\\
6       	0.00410882	\\
7       	0.0033978	\\
8       	0.00285528	\\
9       	0.00247513	\\
10	0.00218369	\\
11	0.00193666	\\
12	0.00172894	\\
13	0.00156918	\\
14	0.00142644	\\
15	0.00129879	\\
16	0.00118438	\\
17	0.00108546	\\
18	0.000998977	\\
19	0.000919474	\\
20	0.000846397	\\
21	0.000779223	\\
22	0.000717465	\\
23	0.000660673	\\
24	0.000608435	\\
25	0.000560376	\\
26	0.000516152	\\
27	0.000475448	\\
28	0.000438172	\\
29	0.000403993	\\
30	0.000372506	\\
31	0.000343439	\\
32	0.000316613	\\
33	0.000291862	\\
34	0.000269029	\\
35	0.00024797	\\
36	0.00022855	\\
37	0.000210643	\\
38	0.000194134	\\
39	0.000178914	\\
40	0.000164884	\\
41	0.000151952	\\
42	0.000140032	\\
43	0.000129045	\\
44	0.000118919	\\
45	0.000109587	\\
46	0.000100987	\\
47	9.30606e-05	\\
48	8.57562e-05	\\
49	7.90247e-05	\\
50	7.28214e-05	\\
    };
    \addlegendentry{$\nu=0.025$}
  \end{semilogyaxis}
\end{tikzpicture}}
      &
      \parbox{0.4\textwidth}{
        \textbf{bold solid lines:} averaging scheme\\
        \textbf{lines with marks:} forward simulation\\

        Number of cycles $n$ required to reduce the periodicity error
        to 
        \[
        \|\vt\big(nP\big)-\vt\big((n-1)P\big)\|<10^{-8}
        \]
        or 50 cycles exceeded. 

        Parameters that are not varies are set to $L=2$, $\nu=0.1$, $P=1$.
      }
    \end{tabular}
  \end{center}
  \caption{Robustness of the forward simulation and the averaging
    scheme with respect to different parameters:
    the domain 
    size $L$ (top/left), the period $P$ (top/right), the viscosity 
    $T$ (bottom/left). No effects of variations in the mesh size $h$
    or time step size $k=P/N$ are observed.}
  \label{fig:stokes}
\end{figure}
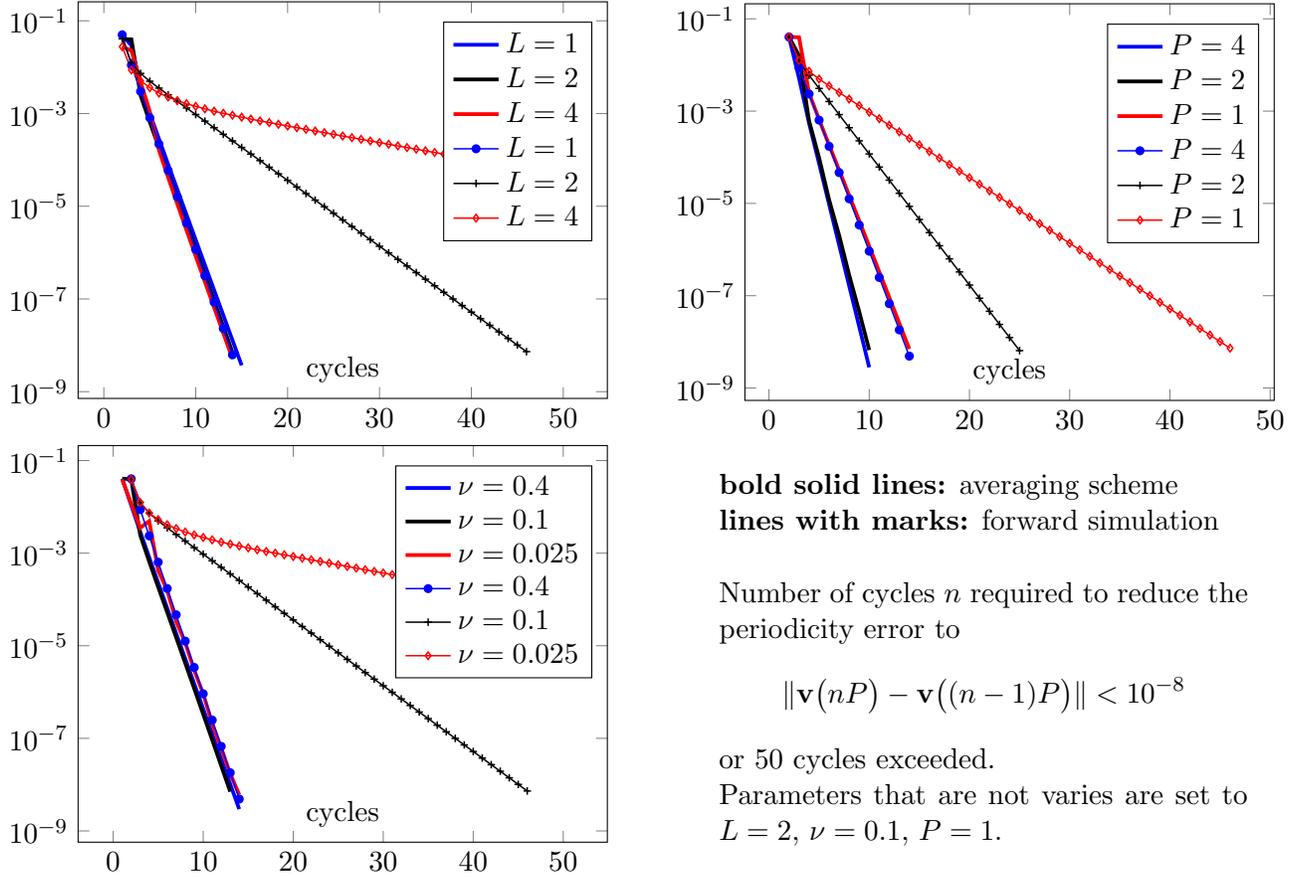

\begin{table}[t]
\begin{center}
  \begin{tabular}{cc|cc|cc}
    \toprule
    \multicolumn{6}{l}{Forward Simulation}\\ \midrule
    $L$ & $\sigma$ &$P$ & $\sigma$ &  $\nu$& $\sigma$ \\
    \midrule
    1&  0.27 &1 & 0.72 & 0.1 & 0.72 \\
    2& 0.72 &2 & 0.52 & 0.05 & 0.85  \\
    4& 0.92 &4 & 0.27 & 0.025& 0.92 \\
    \bottomrule
  \end{tabular} \hspace{1cm}
  \begin{tabular}{cc|cc|cc}
    \toprule
    \multicolumn{6}{l}{Averaging Scheme}\\\midrule
    $L$ & $\sigma$ &$P$ & $\sigma$ & $\nu$& $\sigma$ \\
    \midrule
    1& 0.29 & 1 & 0.28 & 0.1 &0.28  \\
    2& 0.28 & 2 & 0.15 & 0.05& 0.29 \\
    4& 0.28 & 4 & 0.13 & 0.025 &0.29  \\
  \bottomrule
  \end{tabular}
\end{center}
\caption{Convergence rate $\sigma = \|\vt^{(l)}_0-\vt^{(l-1)}_0\|/
  \|\vt^{(l-1)}_0-\vt^{(l-2)}_0\|$. Left: forward simulation. Right:
  averaging scheme. Variation of the domain size $L$,
  the period $P$ and the viscosity $\nu$. }
\label{tab:1}
\end{table}

We start by analyzing the dependency of the two approaches
``forward simulation'' \textbf{(F)} and  ``averaging scheme''
\textbf{(A)} on the domain size $L$, which enters the Poincar\'e
constant and the smallest eigenvalue, the viscosity $\nu$ that directly
influences the reduction rate~(\ref{S5}), further the period length
$P$. The results are shown in Fig.~\ref{fig:stokes}.
We also tested the robustness versus changes in the temporal step
size $k=P/N$ and the mesh size $h$. There was however no influence of
these parameters on the iteration counts, neither in the case of the
forward simulation nor for the averaging scheme. 

Enlarging the domain has a dramatic effect on the forward simulation
\textbf{(F)} as it is shown in the upper/left plot of
Figure~\ref{fig:stokes}. The rate of convergence strongly changes. For
$L=4$ or $L=8$ the required tolerance of $10^{-8}$ could no be
reached within the allowed 50 cycles. In contrast, the averaging
scheme~\textbf{(A)} is very robust. We see a slight
dependency of its convergence rate on the parameter $P$.

In Table~\ref{tab:1} we indicate the convergence rates of the forward
simulation and the averaging scheme in dependency of the parameters,
i.e. the experimental reduction rate $\sigma$ defined by
\[
\sigma^{(l)}:=\frac{\|\vt^{(l)}_0-\vt^{(l-1)}_0\|}{
\|\vt^{(l-1)}_0-\vt^{(l-2)}_0\|}.
\]
The rate of the averaging scheme is very robust. All
test cases show $\sigma<0.29$, which is the limit shown in the
continuous case of Lemma~\ref{lemma:ode}.  In contrast, the simple
forward iteration shows deteriorating convergence rates for smaller
viscosities and period lengths and larger domains and behaves like
\[
\sigma = {\cal O}\left(\exp\left(-\frac{\nu P}{L^2}\right)\right)
\]
which is expected from the theoretical analysis,
see~(\ref{perioddecay}).

The results for the averaging scheme suggest that the analysis in the
discrete case in Lemma~\ref{lemma:ode:disc} is not sharp. 
The bound $0.29$ for the convergence rate identified in the continuous
case in Lemma~\ref{lemma:ode} also appears to be valid in the discrete
setting.

\subsection{Robustness with regard to the Reynolds number}\label{sec:num:ns}

As second test case we study the Navier-Stokes flow in an annulus with
outer radius $R=5$ and inner radius $r=0.5$, see
Fig.~\ref{fig:ns:1}. On both rings we drive the flow by a Dirichlet
condition. While the outer ring is rotating, an inflow/outflow 
condition with an oscillating direction is prescribed on the inner
ring. The maximum velocity on the boundary reaches
$|\vt|=0.5$. The period is chosen as $P=1$ and we choose $N=20$ time
steps for each cycle. Biquadratic finite elements on an isoparametric
mesh with a total of  $N_{dofs} = 49\,920$
degrees of freedom are used. Like in the previous test case we could
not identify any dependency of the convergence rates on the temporal
step size $k=1/N$ or the number of spatial unknowns $N_{dofs}$. 

In Figure~\ref{fig:ns:2} we show the required number of steps for
different viscosities $\nu$ in order to investigate the influence of the
nonlinearity. With $|\vt|=1/2$, $diam(\Omega)=10$ we compute the Reynolds number
as
\[
Re = \frac{|\vt|R}{\nu} = \frac{5}{\nu}. 
\]
Variation of the viscosity from $\nu=1$ to $\nu\approx 0.0078$
corresponds to a range of Reynolds numbers from $Re=5$ to
$Re=640$. For smaller viscosities we could not identify any periodic
solution. In the case of very large viscosity parameters $\nu=1$
there is no benefit of the averaging scheme \textbf{(A)}. 

\begin{figure}[t]
  \begin{center}
    \begin{minipage}{0.26\textwidth}
      \includegraphics[width=\textwidth]{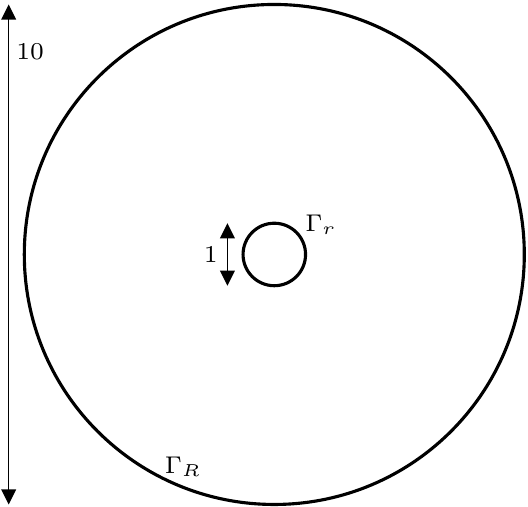}
    \end{minipage}
    \begin{minipage}{0.56\textwidth}\small
      \[
      \begin{aligned}
        \vt_r &= \frac{1}{2} \begin{pmatrix}
          \cos(2\pi t/P)\\
          \sin(2\pi t/P)
        \end{pmatrix}\\ \\
        \vt_R &= \frac{sin(2\pi t/P)}{2 R} \begin{pmatrix}
          -y\\ x
        \end{pmatrix}\\ \\
        r&=\frac{1}{2},\quad R=5,\quad P=1\\
      \end{aligned}
      \]
    \end{minipage}
  \end{center}
  \caption{Configuration of the second Navier-Stokes test case.
    The flow is driven by periodic Dirichlet conditions on both
    boundaries, the inner circle $\Gamma_r$ with radius $r=0.5$ and
    the outer circle $\Gamma_R$ with radius $R=5$. }
  \label{fig:ns:1}
\end{figure}

Increasing the Reynolds number causes a significant increase of
required iterations in the case of the forward simulation \textbf{(F)}
while the necessary iterations for the averaging scheme \textbf{(A)}
stays constant until we leave the regime of periodic solutions. In
between the savings are substantial.

\begin{figure}[t]
  \begin{center}
    \begin{minipage}{0.6\textwidth}
      \setlength{\figureheight}{0.5\textwidth}
      \setlength{\figurewidth}{0.8\textwidth} 
      \begin{tikzpicture}
  \begin{semilogxaxis}[
      width=\figurewidth, height=\figureheight,
      scale only axis,
      x label style={anchor=north, below=-12mm},
      xlabel={Reynolds number},
      ylabel=cycles,
      mark options={solid},
      title style={font=\bfseries},
      ymin=-30,
      legend style={at={(0,0.95)},xshift=0.2cm,anchor=north west,nodes=right}]
    \addplot[color=black,mark=*,solid,line width=0.5mm] table[row sep=crcr]{
      5 22  \\
      10  21 \\
      20 27  \\
      40 43  \\
      80  73 \\
      160 122\\
      320 198\\
      640 295\\
    };
    \addlegendentry{forward scheme \textbf{(F)}}
    \addplot[color=black,mark=triangle,solid,line width=0.5mm] table[row sep=crcr]{
      5 22\\
      10  20\\
      20 16\\
      40 11\\
      80  8\\
      160 8\\
      320 10\\
      640 15\\
    };
    \addlegendentry{averaging scheme \textbf{(A)}}
  \end{semilogxaxis}
\end{tikzpicture}
    \end{minipage}\hspace{0.05\textwidth}
    \begin{minipage}{0.29\textwidth}\small
      ~\vspace{0.4cm}
      \begin{tabular}{cccc}
        \toprule
        $\nu$ & $Re$ & \textbf{(F)} & \textbf{(A)}\\
        \midrule
        $2^{-0}$&5&22 &22\\
        $2^{-1}$&10&21 &20\\
        $2^{-2}$&20&27 &16\\
        $2^{-3}$&40&43 &11\\
        $2^{-4}$&80&73 &8\\
        $2^{-5}$&160&122&8\\
        $2^{-6}$&320&198&10\\
        $2^{-7}$&640&295&15\\
        \bottomrule
      \end{tabular}
    \end{minipage}

  \end{center}
  \caption{Required number of cycles for the averaging scheme and the
    forward iteration to reach the tolerance error 
    $10^{-8}$ depending on the Reynolds number $Re$. }
  \label{fig:ns:2}
\end{figure}
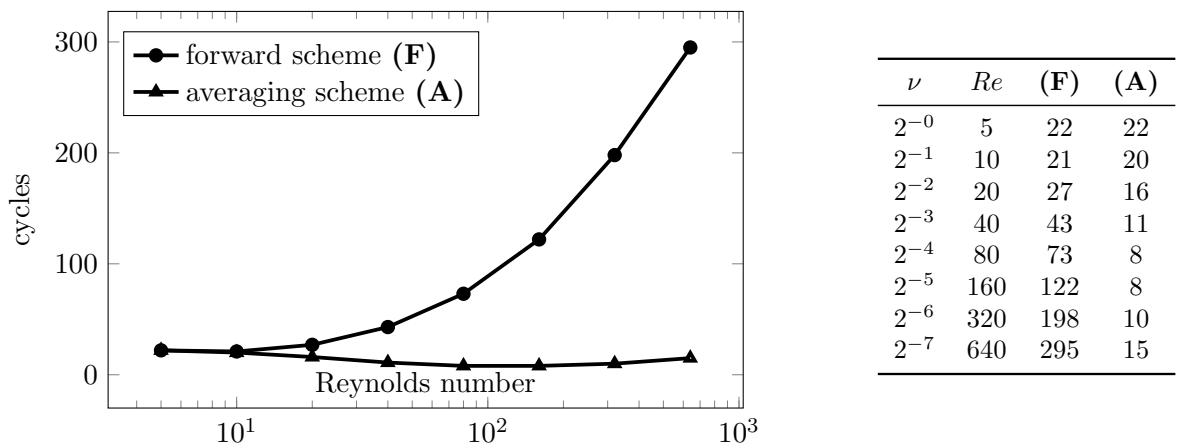

\section{Conclusion}\label{sec:conclusion}

We have presented an acceleration scheme for the computation of
periodic solutions to the Stokes and Navier-Stokes equations. In the
linear Stokes case we could show robust convergence of the scheme with
a rate that does not depend on the problem parameters or the
eigenvalues of the Stokes operator. Applied to the Navier-Stokes
equations numerical tests predict the same robustness and efficiency
of the algorithm. The only numerical overhead of the proposed
algorithm is the computation of one stationary averaged problem in
every cycle of the dynamic process. Depending on the problem data,
which strongly effects the decay rate of direct forward simulations, 
we get a significant speed-up.

The proof of convergence is based on the linearity and symmetry of the
Stokes operator. The extension to the nonlinear Navier-Stokes
equations will call for a different approach. 

An interesting but open extension of the averaging scheme will be the
application to problems with unknown periodicity. A possible
application is the laminar vortex shedding of the flow around an
obstacle. Here, predictions of the frequency are available with the
Strouhal number, the exact value however is depending on the specific
configuration, in particular on the geometry. To tackle such problems
we aim at the combination of the averaging scheme for obtaining
initial values $\vt_0$ with an optimization approach to identify the
period length $P$. The difficulty of such settings is the higher
Reynolds number that is usually involved. There is only a small
regime, where the solution is nonstationary with a clear
periodic-in-time solution. 

Finally, the proposed scheme allows to accelerate several problems
where the computation of cyclic states is an algorithmic sub-task such
as temporal multi-scale
problems~\cite{FreiRichterWick2016,FreiRichter2019,RichterMizerski2020}. 

\section*{Acknowledgement}

The author acknowledges support by the Federal Ministry of Education and
Research of Germany (project number 05M16NMA), as well as funding by
the Deutsche Forschungsgemeinschaft (DFG, German Research Foundation)
- 314838170, GRK 2297 MathCoRe and in project 411046898. 


\end{document}